\documentclass[a4paper]{amsart}

\usepackage[alphabetic]{amsrefs}

\usepackage{amssymb,amscd,amsmath}
\usepackage{a4wide,graphicx,nicefrac}
\usepackage{todonotes}
\usepackage[shortlabels]{enumitem}
\usepackage{mathrsfs}
\usepackage{hyperref}
\usepackage{ourbib}
\usepackage[all]{xy}
\usepackage{callouts}

\setlist{leftmargin=8mm}

\hypersetup{
    colorlinks,
    linkcolor={red!50!black},
    citecolor={blue!50!black},
    urlcolor={blue!80!black}
}

\newcommand{\R}{\mathbb{R}}
\newcommand{\reals}{\mathbb{R}}

\renewcommand{\phi}{\varphi}

\newcommand{\myicon}{$\,\,\,\triangleright$}

\newcommand{\abs}[1]{\left\lvert#1\right\rvert}

\DeclareMathOperator{\ind}{\mathrm{ind}}
\DeclareMathOperator{\id}{\mathrm{id}}

\newtheorem{theorem}{Theorem}

\newtheorem{proposition}[theorem]{Proposition}

\theoremstyle{definition}
\newtheorem{remark}[theorem]{Remark}
 
\newtheorem{example}[theorem]{Example} 
\newtheorem{discussion}[theorem]{Discussion}

\begin{document}

\title[Remarks on a paper by J. Wang, Z. Xie and G. Yu]{Remarks on the paper {\em On Gromov's dihedral extremality and rigidity conjectures} by Jinmin Wang, Zhizhang Xie and Guoliang Yu}  
\author{Christian B\"ar}
\address{Universit\"at Potsdam, Institut f\"ur Mathematik, 14476 Potsdam, Germany}
\email{\href{mailto:cbaer@uni-potsdam.de}{cbaer@uni-potsdam.de}}
\urladdr{\url{https://www.math.uni-potsdam.de/baer}}

\author{Bernhard Hanke}
\address{Universit\"at Augsburg, Institut f\"ur Mathematik, 86135 Augsburg, Germany}
\email{\href{mailto:hanke@math.uni-augsburg.de}{hanke@math.uni-augsburg.de}}
\urladdr{\url{https://www.math.uni-augsburg.de/diff/hanke}}

\author{Thomas Schick} 
\address{Universit\"at G\"ottingen, Mathematisches Institut, 37073 G\"ottingen, Germany} 
\email{\href{mailto:thomas.schick@math.uni-goettingen.de}{ thomas.schick@math.uni-goettingen.de}} 
\urladdr{\url{https://www.uni-math.gwdg.de/schick}}

\begin{abstract} 
 Version~2 of the article ``On Gromov's dihedral extremality and rigidity conjectures'' by Jinmin Wang, Zhizhang Xie and Guoliang Yu  makes a number of claims for self-adjoint extensions of Dirac type operators on manifolds with corners under local boundary conditions. 
We construct a counterexample to an index computation in that paper which affects the proof of its main result stating a generalisation of Gromov's dihedral extremality conjecture.

\end{abstract}



\maketitle

Consider the manifold with corners 
 \begin{equation*}
A:= [-2,2]^2\setminus (-1,1)^2\subset  \reals^2.
\end{equation*}
together with the Euler characteristic operator $d+d^*\colon  \Omega^{*} (A) \to \Omega^{*}(A)$ on smooth differential forms with grading by even/odd differential forms. 
Let $D_i$ (for initial operator) be the restriction of $d + d^*$ to  forms supported  in the complement of the vertices of $A$ and  subject to absolute boundary conditions.

\begin{remark} Recall that (by definition) a smooth differential form $\omega \in \Omega^*(A)$ supported away from the vertices satisfies {\em absolute boundary conditions} if
  \begin{equation*}
     \iota^*(*\omega)=0
   \end{equation*}
 where $*\omega$ is the Hodge star of $\omega$ and $\iota\colon\partial  A\to A$ is the inclusion of the boundary. 
Equivalently, the contraction of $\omega$ with vectors normal  to the boundary is zero.
\end{remark} 

Explicitly, a differential form $\omega \in \Omega^*(A)$ in degrees $0$, $2$ and $1$  lies in ${\rm dom}(D_i)$, if $\omega$ is, respectively, 
    \begin{enumerate}[\myicon] 
   \item a smooth function $f \colon A \to \R$ supported  away from the vertices, 
   \item a smooth $2$-form $f(x,y)dx \wedge dy \in \Omega^2(A)$ where $f$ is supported away from the  vertices and  vanishes on $\partial A$, 
  \item  a smooth $1$-form $fdx+gdy \in \Omega^1(A)$ where  $f$ and $g$ are supported away from the vertices, $f$ vanishes on the vertical boundary edges
  \[ 
      \partial_v A  = \big( \{\pm 2\}\times [-2,2] \big)  \cup   \big( \{\pm 1\}\times [-1,1] \big)  \subset \partial A
    \]
 and $g$ vanishes on the horizontal boundary edges 
\[
    \partial_h A = \big( [-2,2] \times \{\pm 2\} \big) \cup \big( \{\pm 1\} \times [-1,1] \big) \subset \partial A .
  \]
\end{enumerate} 

We obtain a densely defined unbounded symmetric and closable operator 
\[
    D_i=d+d^* \colon L^2 \Omega^*(A) \supset  {\rm dom}(D_i) \to  L^2  \Omega^*(A) . 
\]
It coincides with the operator $D_B$ appearing in \cite{WXYv2}*{Chapter 3} for $N=M:= A$ and $f:= \id_A$, compare \cite{WXYv2}*{Proposition A.2}. 

Following \cite{WXYv2}*{Proposition  3.6} we now consider the self-adjoint extension of $D_i$ given by 
\[
   D_e  := \left( \begin{array}{cc} 0  & (D_i)_{\rm min} \\   (D_i)_{\rm max} & 0 \end{array} \right), 
\]
that is, restricted to $1$-forms we use its minimal extension $(D_i)_{\rm min}$ and on even degree forms its maximal extension $(D_i)_{\rm max}$ (which contains the minimal extension).

\begin{proposition} \label{main} The operator $D_e$ has the following properties: 
\begin{enumerate}[(i)] 
\item Restricted to $0$-forms its kernel is at least $1$-dimensional, 
\item restricted to $1$-forms its kernel is $0$. 
\end{enumerate} 
\end{proposition} 

\begin{remark} \begin{enumerate}[\myicon]
\item This contradicts \cite{WXYv2}*{Claim on page 34} which for $N = M:= A$ and $f:= \id_A$ states that $D_e$ is Fredholm with index equal to the Euler characteristic of $A$ (which is equal to $0$).
In particular, $\ind(D_e)$ is not equal to the index of the Euler characteristic operator with absolute boundary conditions on a smooth annulus  (which is equal to the Euler characteristic of the annulus), contrary to what is claimed in \cite{WXYv2}*{Section 4}. 
\item With a slighly more elaborate proof one can show that $D_e$ is Fredholm with index $1$. 
\end{enumerate}
\end{remark} 

\begin{proof}[Proof of Proposition \ref{main}]  
\underline{Part (i):} We prove that the constant functions on $A$ are in the minimal domain of $D_i$ (and therefore also in the maximal domain), which implies the assertion.

Since the question is local, we can concentrate on a single vertex $V$ of $A$ which we place in the origin. 
Near $V$ the given manifold $A$ is just a cone with a certain angle $2 \beta$ equal to $\pi/2$ or $3\pi/2$. 
(For the following computation the value of the angle $\beta$ is irrelevant).

Around $V$ we use polar coordinates $(r,\phi)$ and on the ball $B_1(0)$ we approximate  the constant function with value $1$ in $H^1$ by functions
 \begin{equation*}
f_\epsilon^\alpha(r,\phi):= 
 \begin{cases}   1-\epsilon^\alpha r^{-\alpha}; & \epsilon <r<1\\
     0; & r<\epsilon
\end{cases}
\end{equation*}
for $\alpha>0$ small and $\epsilon>0$ small. 
Note that $f_\epsilon^{\alpha}$ is continuous and smooth outside $\{r=\epsilon\}$.

We have on $\{r>\epsilon\}$ that
\begin{equation*}
\abs{\nabla f_\epsilon^\alpha} = \abs{\partial_r f_{\epsilon}^\alpha} =
\epsilon^\alpha \alpha r^{-\alpha-1}. 
\end{equation*}
Therefore (integrating in polar coordinates) we obtain 
\begin{align*}
 |\nabla f_{\epsilon}^{\alpha}|_{L^2(B_1(0))} =    \int_{B_1(0)} \abs{\nabla f_{\epsilon}^\alpha}^2 &= \int_\epsilon^1  \epsilon^{2\alpha} \alpha^2 r^{-2\alpha-2} 2\beta r \,dr = 2\beta\epsilon^{2\alpha}\alpha^2\frac{1}{-2\alpha} r^{-2\alpha}|_\epsilon^1 \\
    &= \beta\alpha - \beta \alpha \epsilon^{2\alpha} .
  \end{align*}

Furthermore $1-f_\epsilon^\alpha$ is positive and  is bounded above by $1$ on $\{0<r<\sqrt{\epsilon}\}$ and by $\epsilon^\alpha(\sqrt{\epsilon})^{-\alpha} = \epsilon^{\alpha/2}$ on $\{r>\sqrt{\epsilon}\}$. 
Therefore
\begin{equation*}
  \abs{1-f_\epsilon^\alpha}^2_{L^2(B_1(0))} \le \beta \epsilon + \beta \epsilon^\alpha. 
\end{equation*}

For $\alpha>0$ we now set $\epsilon(\alpha)  :=  \alpha^{1/\alpha}$ and obtain $\lim_{\alpha \to 0} \epsilon(\alpha)  = 0$, $\lim_{\alpha \to 0} \epsilon(\alpha)^{\alpha} \to 0$ and consequently 
\begin{equation*}
  \abs{1-f_{\alpha^{1/\alpha}}^\alpha}^2_{H^1} = \int_{B_1(0)} \abs{\nabla
    f_{\alpha^{1/\alpha}}^\alpha}^2 +
  \abs{1-f_{\alpha^{1/\alpha}}^\alpha}_{L^2} \xrightarrow{\alpha\to 0} 0.
\end{equation*}

\underline{Part (ii):} 
Assume that $\omega \in L^2 \Omega^1(A)$ is  in the kernel of the minimal extension of $D_i$.
By definition, there exist smooth $1$-forms $\omega_n:= f_n dx+g_ndy$ on $A$ with support away from the vertices, satisfying absolute boundary conditions and such that 
\[
     \omega_n\xrightarrow{L^2}\omega , \quad (d+d^*)\omega_n\xrightarrow{L^2}0 . 
\]
We have
  \begin{equation*}
    d\omega_n =( -\partial_y f_n + \partial_x g_n)dx\wedge dy; \quad d^*\omega_n = \partial_xf_n+\partial_y g_n 
  \end{equation*}
and hence 
 \begin{align*} 
     \abs{\omega_n}^2_{L^2} &= \int_A (-\partial_yf_n+\partial_x g_n)^2 + (\partial_xf_n+\partial_y g_n)^2 \\
     &= \int_A (\partial_yf_n)^2 +(\partial_x g_n)^2 -2(\partial_yf_n)(\partial_x g_n) + (\partial_xf_n)^2+(\partial_yg_n)^2 + 2(\partial_xf_n)(\partial_y g_n).
   \end{align*} 
  Now we apply partial integration to $\partial_yf_n\cdot \partial_xg_n$ and $\partial_x f_n\cdot\partial_y g_n$ (which is possible without problems because $f_n,g_n$ are smooth) and  obtain (with the obvious appropriate domains of integration when applying Fubini's theorem)
  \begin{align*}
      - \int_A (\partial_x f_n)\cdot(\partial_y g_n) &= -\int
      \left(\int \partial_x f_n\cdot \partial_y g_n\,dx\right)\,dy\\
      &
      \stackrel{f_n=0\text{ on }\partial_v A}{=} \int  \left( \int f_n\cdot\partial_y\partial_x g_n\,dx\right)\,dy = -\int_A   f_n\partial_y\partial_x g_n.
    \end{align*}
 Similarly,
   \begin{equation*}
       \int_A(\partial_yf_n)\cdot(\partial_xg_n) =   \int\left(\int \partial_yf_n \cdot \partial_x g_n\,dy\right)\,dx\\
       \stackrel{\partial_x g_n=0 \text{ on }\partial_h A}{=} - \int_A f_n \partial_y\partial_x g_n.
   \end{equation*}
Combined with the above computation, we obtain
\begin{equation*}
  \abs{(d+d^*)\omega_n}^2_{L^2} = \abs{\partial_xf_n}_{L^2}^2 + \abs{\partial_y f_n}^2_{L^2} + \abs{\partial_x g_n}^2_{L^2} + \abs{\partial_y g_n}^2_{L^2} \xrightarrow{n\to\infty} 0.
\end{equation*}

We now can use classical Sobolev theory on $\reals^2$ without having to worry about the vertices. 
Even if we don't want to define and use Sobolev spaces on the non-smooth domain $A$, certainly  $\omega=fdx+gdy\in H^1_{loc}\Omega^1(A^\circ)$ with $\partial_xf=0$, $\partial_yf=0$, $\partial_xg=0$, and $\partial_yg=0$. 
Consequently, $\omega$ is constant (and smooth).

The boundary conditions for $f_n$ and $g_n$ imply that $f=0$ as $L^2$-function on the vertical boundary edges and that $g=0$ as $L^2$-function on the horizontal boundary edges, therefore $\omega=0$, as claimed.
\end{proof}

\begin{discussion}
One might wonder how this relates to the fact that $A$ (or rather its interior) is conformally equivalent to a suitable smooth annulus, on which we have a non-trivial harmonic $1$-form $\omega$ which satisfies absolute boundary conditions and which is (of course) in $L^2$, compare \cite{WXYv2}*{Example 3.7.}. 

It is true that the pullback $f^*\omega$ with the conformal isomorphism $f$ between the interior of the annulus and the interior of $A$ produces a $L^2$-harmonic $1$-form on $A$.
 It seems also to be true that $f^*\omega$ satisfies absolute boundary conditions at the regular part of the boundary.
However, the behaviour of $f^*\omega$ near the singularities is not controlled. 
In particular, there is no reason (and it is not true) that it is an $L^2$-limit of smooth forms $\omega_n$ with support off the singularities such that also $(d+d^*)\omega_n$ $L^2$-converges to $0$. 
Outside middle degree,  $f^*$ is far from being an $L^2$-isometry on forms or functions,  and it does not commute with $d+d^*$ in general.
\end{discussion}

The following counterexample to \cite{WXYv2}*{Theorem 1.6.} was communicated to us  by Rudi Zeidler. 

\begin{example}
Let $P$ be the non-convex $2$-dimensional polyhedron obtained from the square $[-2,2]^2$ by removing an open regular pentagon. 
Take as a second non-convex polyhedron $Q$ the complement of $[-2,2]^2$ with an open equilateral triangle in the interior removed. 
  
Now construct a corner  map $P \to Q$ (see \cite{WXYv2}*{Definition 1.5}) as follows: 
\begin{enumerate}[\myicon] 
  \item Near the  outer (square) boundary $\partial A \subset \partial P$ take the identity, 
  \item send $3$ consecutive edges $AB$, $BC$, $CD$ of the inner pentagon $ABCDE \subset P$ by affine linear maps to the three edges $A'B'$, $B'C'$, $C'A'$ of the inner triangle $A'B'C' \subset Q$, 
  \item send the remaining vertex $E$ (like $A$ and $D$) to the vertex $C'$ and the edges $DE$ and $AE$  by ``fold'' maps to the edges $A'B'$ and $A'C'$. 
  (The local model for such a fold map could just be $[-1,1]\to [-1,1]; t\mapsto t^2$.) 
  \item Extend this map arbitrarily to a corner map on $P$. 
\end{enumerate} 

Next remove another small open convex polyhedron (e.g.~a triangle) near the outer (square) boundaries of both  $P$ and $Q$, to produce $P'$ and $Q'$, which then have Euler characteristic $-1$. 
By restriction of the given map $P \to Q$ we obtain a corner map $f' \colon P' \to Q'$. 

For $r \gg 0$ we denote by $rP'$ the $r$-scaling of $P$. 
Composition of multiplication by $1/r$ and the map $f'$ yields a corner map $f \colon r P'\to Q'$. 

If $r$ is chosen large enough, the map $f$  satisfies all the conditions of \cite{WXYv2}*{Theorem 1.6.}, but Conclusion (iii) of that theorem doesn't hold. 
 \end{example}

\begin{discussion} We believe that the last equality in the estimate
\[
    - \int_{\overline{C}_{ij,r}}  k_r \langle e_1 , f_* \overline{e}_1\rangle_M \cdot \langle \overline{c}_{\partial}(\overline{e}_1) \otimes c_{\partial} (e_1) \phi, \phi \rangle d \overline{s} \geq - \int_{C_{ij,r}}  | k_r | \cdot \langle \phi , \phi \rangle ds = - ( \pi - \theta_{ij}) |\phi(v_{ij})|^2 + o(1) 
\]
appearing in the proof of  \cite{WXYv2}*{Proposition 2.6} (see middle of page 19) requires the additional assumption that all $\theta_{ij} \leq \pi$, since otherwise $\pi - \theta_{ij}$ becomes negative, while $|k_r|$ is non-negative. 
This flaw may be remedied by adding the assumption to \cite{WXYv2}*{Proposition 2.6} that all dihedral angles appearing in $M$ are less than or equal to $\pi$, which  seems to be in accordance with the given non-negativity assumption on the second fundamental form of $\partial M$. 

However, we do not see where such a convexity assumption on dihedral angles may help for the index theoretic argument  outlined in  \cite{WXYv2}*{Sections 3 and 4}. 
\end{discussion}

\begin{discussion}
  The most delicate part seems to be the analysis of the local boundary value
  problems on singular domains, which is a classical and very delicate
  field of study with few general results and many unexpected features.

  We decided to make this note publicly available to increase awareness of
  these problems, which might help the community to avoid technical pitfalls.

  The authors of \cite{WXYv2}, after having been informed by us about the
  counterexamples, have revised their paper a number of times.
  The current version~6 of their paper \cite{WXYv2v6} is still under verification.

\end{discussion}

\end{document}